\documentclass[10pt]{amsart}
\usepackage{amsmath, amssymb}
 \usepackage{mathrsfs}
 \usepackage{enumerate}
\newcommand{\no}[1]{#1}
\renewcommand{\no}[1]{}
\no{\usepackage{times}\usepackage[subscriptcorrection, slantedGreek, nofontinfo]{mtpro}
\renewcommand{\Delta}{\upDelta}}
\usepackage{color}

\date{\today}
\newtheorem{theorem}{Theorem}[section]
\newtheorem{proposition}{Proposition}[section]
\newtheorem{lemma}{Lemma}[section]

\newtheorem{corollary}{Corollary}[section]

\theoremstyle{remark}
\newtheorem{remark}{Remark}[section]

\numberwithin{equation}{section}

\newcommand{\pd}{\partial}

\newcommand{\mydiv}{{\text{div}}}

\newcommand{\R}{{\mathbb R}}
\newcommand{\N}{{\mathbb N}}

\newcommand{\bel}{\begin{equation} \label}
\newcommand{\ee}{\end{equation}}
\newcommand{\bea}{\begin{eqnarray}}
\newcommand{\eea}{\end{eqnarray}}
\newcommand{\beas}{\begin{eqnarray*}}
\newcommand{\eeas}{\end{eqnarray*}}

\title[Isospectral potentials]{Heat trace asymptotics and compactness of isospectral potentials for the Dirichlet Laplacian}

\author[Mourad Choulli]{Mourad Choulli\S}
\address{\S Institut \'Elie Cartan de Lorraine, UMR CNRS 7502, Universit\'e de Lorraine, B.P. 70239, 54506 Vandoeuvre-l\`es-Nancy Cedex, France}
\email{mourad.choulli@univ-lorraine.fr}

\author[Laurent Kayser]{Laurent Kayser\P}
\address{\P Institut \'Elie Cartan de Lorraine, UMR CNRS 7502, Universit\'e de Lorraine, B.P. 70239, 54506 Vandoeuvre-l\`es-Nancy Cedex, France}
\email{laurent.kayser@univ-lorraine.fr}

\author[Yavar Kian]{Yavar Kian\dag}
\address{\dag Centre de Physique Th\'eorique, UMR CNRS 7332, Universit\'e d'Aix-Marseille, Universit\'e du Sud-Toulon-Var, CNRS-Luminy, 13288 Marseille, France}
\email{yavar.kian@univ-amu.fr}

\author[Eric Soccorsi]{Eric Soccorsi\ddag}
\address{\ddag Centre de Physique Th\'eorique, UMR CNRS 7332, Universit\'e d'Aix-Marseille, Universit\'e du Sud-Toulon-Var, CNRS-Luminy, 13288 Marseille, France}
\email{eric.soccorsi@univ-amu.fr}

\date{}

\begin{document}

\begin{abstract}
Let $\Omega$ be a $C^\infty$-smooth bounded domain of $\mathbb{R}^n$, $n \geq 1$, and let the matrix ${\bf a} \in C^\infty (\overline{\Omega};\R^{n^2})$ be symmetric and uniformly elliptic. We consider the $L^2(\Omega)$-realization $A$ of the operator $-\mydiv ( {\bf a} \nabla \cdot)$ with Dirichlet boundary conditions. We perturb $A$ by some real valued potential
$V \in C_0^\infty (\Omega)$ and note $A_V=A+V$. We compute the asymptotic expansion of $\mbox{tr}\left( e^{-t A_V}-e^{-t A}\right)$ as $t \downarrow 0$ for any matrix ${\bf a}$ whose coefficients are homogeneous of degree $0$. In the particular case where $A$ is the Dirichlet Laplacian in $\Omega$, that is when ${\bf a}$ is the identity of $\R^{n^2}$, we make the four main terms appearing in the asymptotic expansion formula explicit and prove that $L^\infty$-bounded sets of isospectral potentials of $A$ are $H^s$-compact for $s <2$.


\end{abstract}

\maketitle

\tableofcontents


\section{Introduction}
\label{sec-intro}

In the present paper we investigate the compactness issue for isospectral potentials sets of the Dirichlet Laplacian by means of heat kernels asymptotics.

\subsection{Second order strongly elliptic operator} 
Let $\mathbf{a}=(a_{ij})_{1\leq i,j \leq n}$ be a symmetric matrix of $\R^{n^2}$, $n \geq 1$, with coefficients in $C^\infty (\mathbb{R}^n)$. We assume that $\mathbf{a}$ is uniformly elliptic, in the sense that there is a constant $\mu \geq 1$ such that the estimate
\bel{i0}
\mu^{-1} \leq  \mathbf{a}(x) \leq \mu,
\ee
holds for all $x \in \mathbb{R}^n$ in the sense of quadratic forms on $\R^n$.

We consider a bounded domain $\Omega \subset \mathbb{R}^n$, with $C^\infty$ boundary $\pd \Omega$ and introduce the selfadjoint operator $A$ generated in $L^2(\Omega)$ by the closed quadratic form
\bel{i0b}
\mathfrak{a}[u] = \int_{\Omega} a(x) | \nabla u(x) |^2 dx,\ u \in D(\mathfrak{a})=H_0^1(\Omega),
\ee
where $H_0^1(\Omega)$ is the closure of $C_0^\infty(\Omega)$ in the topology of the standard first-order Sobolev space $H^1(\Omega)$. Here $\nabla$ stands for the gradient operator on
$\R^n$. By straightforward computations we find out that $A$ acts on its domain
$D(A)=H^2(\Omega) \cap H_0^1(\Omega)$, as
\bel{i1}
A=-\mydiv(a(x) \nabla \cdot) =  -\sum_{i,j=1}^n \partial_j(a_{ij}(x) \partial_i).
\ee
Let $V \in C_0^\infty (\mathbb{R}^n)$ be real-valued. We define the perturbed operator $A_V=A+V$ as a sum in the sense of quadratic forms. Then we have $D(A_V)=D(A)$ by \cite{RS2}[Theorem X.12]. 

\subsection{Main results}
Put
\bel{i2}
Z_\Omega ^V(t)=\mbox{tr}\left( e^{-t A_V}-e^{-tA}\right),\ t>0.
\ee
Much of the technical work developped in this paper is devoted to proving the existence of real coefficients $c_k(V)$, $k \geq 2$, such that following symptotic expansion
\bel{i3} 
Z_\Omega ^V(t)=t^{-n/2}\left (tc_2(V)+t^{3/2}c_3(V)+\ldots +t^{k/2}c_k(V)+O\left(t^{k/2+1/2}\right) \right),\ t \downarrow 0,
\ee
holds for ${\bf a}$ homogeneous of degree $0$.
In the peculiar case where $\mathbf{a}$ is the identity matrix then \eqref{i3} may be refined, providing
\bel{i4}
Z_\Omega ^V(t)=t^{-n/2}\left (td_1(V)+t^2d_1(V)+\ldots +t^pd_p(V)+O\left(t^{p+1}\right)\right),\  t \downarrow 0,
\ee
where $d_k(V)$, $k \geq 1$, is a real number depending only on $V$.
Moreover we shall see that \eqref{i2}-\eqref{i3} remain valid upon substituting $\R^n$ for $\Omega$ in the definition of $A$ (and subsequently $H^1(\R^n)$ for $H_0^1(\Omega)$ in \eqref{i0b}).

Since $\Omega$ is bounded then the injection $H_0^1(\Omega) \hookrightarrow L^2(\Omega)$ is compact. Thus the resolvent of $A_V$ is a compact operator and the spectrum of
$A_V$ is pure point. Let $\{ \lambda_j^V,\ j \in \N^* \}$ be the non-decreasing sequence of the eigenvalues of $A_V$, repeated according to their multiplicities. We define the isospectral set associated to the potential $V \in C_0^{\infty}(\Omega)$ by
$$
\mbox{Is}(V)=\{ W \in C_0^\infty(\Omega );\ \lambda_k^V=\lambda_k^{W},\ k \in \N^* \}.
$$
The computation carried out in \S \ref{sec-coe} of the coefficients $d_j(V)$ appearing in \eqref{i4}, for $j=1,2,3,4$, leads to the following compactness result.

\begin{theorem}
\label{thm-comp}
Let ${\bf a}$ be the identity of $\R^{n^2}$. Then for all $V \in C_0^{\infty}(\Omega)$ and any bounded subset $\mathcal{B} \subset L^\infty (\Omega )$ such that $V \in \mathcal{B}$, the set $\mbox{Is}(V) \cap \mathcal{B}$ is compact in $H^s(\Omega )$ for each $s \in (-\infty,2)$. 
\end{theorem}

\subsection{What is known so far}
It turns out that the famous problem addressed by M. Kac in \cite{Ka}, as whether one can hear the shape of drum, is closely related to the following asymptotic expansion formula for the trace of $e^{t\Delta _g}$ on a compact Riemannian manifold $(M,g)$: 
\bel{i5}
\mbox{tr}\left( e^{t\Delta _g}\right)=t^{-n/2}\left( e_0+te_1+t^2e_2+\ldots +t^ke_k+O\left(t^{k+1}\right)\right).
\ee
Here $\Delta_g$ is the Laplace-Beltrami operator associated to the metric $g$ and the coefficients $e_k$, $k\geq 0$, are Riemannian invariants depending on the curvature tensor and its covariants derivatives. There is a wide mathematical literature about \eqref{i5}, with many authors focusing more specifically on the explicit calculation of $e_k$, $k \geq 0$. 
This is due to the fact that these coefficients actually provide useful information on $g$ and consequently on the geometry of the manifold $M$. The key point in the proof of \eqref{i5}is the construction of a parametrix for the heat equation $\partial _t-\Delta_g$, which was initiated by S. Minakshisudaram and \AA. Pleijel in \cite{MP}.

The asymptotic expansion formula \eqref{i4} was proved by Y. Colin de Verdi\`ere in \cite{Co} by adaptating \eqref{i5}. An alternative proof, based on the Fourier transform, was given in \cite{BB} by R. Ba\~{n}uelos and A. S\'a Barreto. 
The approach developped in this text is rather different in the sense that \eqref{i3} is obtained by linking the heat kernel of $e^{-tA_V}$ to the one of $e^{-tA}$ through Duhamel's principle. The asymptotic expansion formulae \eqref{i4} and \eqref{i5} are nevertheless quite similar, but, here, the coefficients $d_k$, $k \geq 1$, are stated as integrals over $\Omega$ of polynomial functions in $V$ and its derivatives. This situation is reminiscent of \cite{BB}[Theorem 2.1] where the same coefficients are expressed in terms of the tensor products $\widehat{V}\otimes \ldots \otimes \widehat{V}$, where $\widehat{V}$ is the Fourier transform of the potential $V$. Since the present work is not directly related to the analysis of the asymptotic expansion formula \eqref{i5}, we shall not go into that matter further and we refer to \cite{BGM, Ch, Gi2, Ka, MS} for more details.

As will appear in section \ref{sec-proof}, the proof of the compactness Theorem \ref{thm-comp} boils down to the calculation of the four main terms in the asymptotic expansion formula \eqref{i3}. This follows from the basic identity
$$
\sum_{k\geq 1}e^{-\lambda _k^Vt}=\mbox{tr}\left( e^{-tA_V}\right)=\mbox{tr}\left(e^{-tA}\right)+ Z_\Omega ^V(t),
$$
linking the isospectral sets of $A_V$ to the heat trace of $A$. Compactness results for isospectral potentials associated to the operator $\Delta_g+V$ were already obtained by  Br\"uning in \cite{Br}[Theorem 3] for a compact Riemannian manifold with dimension no gretaer that $3$, and further improved by Donnelly in \cite{Don}. Their approach is based on trace asymptotics borrowed to \cite{Gi1}[Theorem 4.3]. Our strategy is rather similar but the heat kernels asymptotics needed in this text are explicitly computed in the first part of the article.

\subsection{Outline}
Section \ref{sec-pre} gathers several definitions and auxiliary results on heat kernels and trace asymptotics neeeded in the remaining part of the article. The asymptotic formulae \eqref{i3}-\eqref{i4} areestablished in Section \ref{sec-af}. Finally section \ref{sec-proof} contains the proof of Theorem \ref{thm-comp}.

\section{Preliminaries}
\label{sec-pre}

In this section we introduce some notations used throughout this text and derive auxiliary results needed in the remaining part of this paper. 

\subsection{Heat kernels and trace asymptotics}
With reference to the definitions and notations introduced in \S \ref{sec-intro} we first recall from \cite{Ou} that the operator $(-A_V)$, where $V \in C_0^{\infty}(\Omega)$, generates an analytic semi-group $e^{-t A_V}$ on $L^2(\Omega )$. We note
$K^V$ the heat kernel associated to $e^{-t A_V}$, in such a way that the identity
\bel{a1}
\left(e^{-t A_V}f\right)(x)=\int_\Omega K^V(t,x,y)f(y)dy,\ t>0,\ x \in \Omega,
\ee
holds for every $f \in L^2(\Omega)$. Let $M_V$ be the multiplier by $V$. Then we have
\[
e^{-t A_V}=e^{-tA}-\int_0^te^{-(t-s)A}M_Ve^{-s A_V}ds,\ t >0,
\]
from Duhamel's formula. From this and \eqref{a1} then follows that
\bel{a2}
K^V(t,x,y)=K(t,x,y)-\int_0^t\int_\Omega K(t-s,x,z)V(z)K^V(s,z,y)dzds,\ t>0,\ x,y \in \Omega,
\ee
where $K$ denotes the heat kernel of $e^{-tA}$. Upon solving the integral equation \eqref{a2} with unknown function $K^V$ by the successive approximation method, we obtain that
\bel{a3}
K^V(t,x,y)=\sum_{j\geq 0}K^V_j(t,x,y),\ t>0,\ x,y \in \Omega,
\end{equation}
with 
\bel{a3b}
K^V_0(t,x,y)  =  K(t,x,y)\ \mbox{and}\ K^V_{j+1}(t,x,y) = -\int_0^t \int_\Omega K(t-s,x,z)V(z)K^V_j(s,z,y)dsdz\ \mbox{for\ all}\ j \in \N.
\ee
Thus, for each $t>0$ and $x,y \in \Omega$, we get by induction on $j \in \N^*$ that
$$
K_j^V(t,x,y) = (-1)^j \int_{\Omega^n} \int _0^t \int_0^{t_1} \ldots \int_0^{t_{j-1}} \left[ \prod_{i=1}^j K(t_{i-1}-t_i,z_{i-1},z_i)V(z_i) \right] K(t_j,z_j,y) dz^j dt^j,
$$
where $t_0=t$, $z_0=x$, and $d u^j= du_1 \ldots d u_j$ for $u=z,t$.
From this, the following reproducing property
\bel{a5}
\int_\Omega K(t-s,x,z)K(s,z,y)dz=K(t,x,y),\ t>0,\ s \in (0,t),\ x,y \in \Omega,
\ee 
and the estimate $K \geq 0$, arising from \cite{Fr}, then follows that
\bel{a4}
|K_j^V(t,x,y)|\leq \frac{\|V\|^j_\infty t^j}{j!}K(t,x,y),\ t>0,\ x,y \in \Omega,\ j \in \N.
\ee
Therefore, for any fixed $x,y \in \Omega$, the series in the rhs of \eqref{a3} converges uniformly in $t >0$.

Having said that we consider the fundamental solution $\Gamma$  to the equation 
$$ \pd_t -\mydiv({\bf a}(x) \nabla\ \cdot ) = \pd_t-\sum_{i,j=1}^n \partial_j (a_{ij}(x) \partial_i\ \cdot)=0\ \mbox{in}\ \mathbb{R}^n. 
$$
Then there is a constant $c>0$, depending only on $n$ and $\mu$, such that we have
\bel{a6}
\Gamma (t,x,y)\leq (ct)^{-n/2}e^{-c|x-y|^2 \slash t},\ t>0,\ x,y \in \mathbb{R}^n,
\ee
according to \cite{FS}. Further, arguing as in the proof of Lemma \ref{lm-heatker} below, it follows from the maximum principle that
\bel{a7}
0 \leq K(t,x,y) \leq \Gamma(t,x,y),\ t>0,\ x , y \in \Omega.
\ee
Thus, for all fixed $t>0$, the series in the rhs of \eqref{a3} converges uniformly with respect to $x$ and $y$ in $\Omega$ according to \eqref{a4} and \eqref{a6}, and we have
\bel{a8}
\int_\Omega K^V(t,x,x)dx =\sum_{j\geq 0} A_j^V(t)\ \mbox{where}\ 
A_j^V(t)=\int_\Omega K_j^V(t,x,x)dx,\ j \in \N.
\ee
Since $\sigma (e^{-t A_V})=\{ e^{-t\lambda^V_k},\; k\geq 1\}$ from the spectral theorem then
$e^{-t A_V}$ is trace class by \cite{Kat}. On the other hand, $e^{-t A_V}$ being an integral operator with smooth kernel (see e.g. \cite{Da}), we have
\bel{a8b}
\mbox{tr}\left(e^{-t A_V}\right)= \int_\Omega K^V(t,x,x)dx = \sum_{k\geq 1}e^{-t \lambda^V_k},\ t>0.
\ee
Notice that the right identity in \eqref{a8b} is a direct consequence of  Mercer's theorem (see e.g. \cite{Ho}), entailing
$$
K^V(t,x,y)=\sum_{k\geq 1} e^{-t \lambda^V_k} \phi_k^V(x) \times \overline{\phi_k^V(y)},\ t >0,\ x , y \in \Omega,
$$
where $\{ \phi _k^V,\ k \in \N^* \}$ is an orthonormal basis of eigenfunctions $\phi_k^V$ of $A_V$, associated to the eigenvalue $\lambda_k^V$.
Finally, putting \eqref{i2} and \eqref {a8}-\eqref{a8b} together, we find out that
\bel{a8c}
Z_\Omega^V(t) = \sum_{j\geq 1} A_j^V(t),\ t>0.
\ee

\subsection{Estimation of Green functions}
We start with the following useful comparison result:
\begin{lemma}
\label{lm-heatker}
For $\delta >0$ put $\Omega _\delta =\{ x\in \Omega ;\ \mathrm{dist}(x,\pd \Omega)>\delta \}$. Then we have
$$
0 \leq \Gamma (t,x,y)-K(t,x,y)\leq (ct)^{-n/2}e^{-c\delta ^2 \slash t},\ 0 < t \leq \frac{2c\delta ^2}{n},\ x\in \Omega,\ y\in \Omega _\delta,
$$
where $c$ is the constant appearing in the rhs of \eqref{a6}.
\end{lemma}
\begin{proof}
Fix $y\in \Omega_\delta$. Then $u_y(t,x)=\Gamma (t,x,y)-K(t,x,y)$ being the solution to the following initial boundary value problem
$$
\left\{
\begin{array}{ll}
\partial_t u_y(t,x) -\sum_{i,j=1}^n \partial _j(a_{ij}\partial_i u_y(t,x) )=0, &t>0,\ x \in \Omega,
\\
u_y(0,x )=0, & x\in \Omega ,
\\
u_y(t,x) =\Gamma (t,x,y), & t>0,\ x\in \partial \Omega,
\end{array}
\right.
$$
we get from the parabolic maximum principle (see e.g. \cite{Fr}) that
$u_y(t,x)\leq \max_{\substack{z\in \partial \Omega \\ 0<s\leq t}}\Gamma (s,z,y)$ for ae $x \in \Omega$.
Therefore we have
$$
u_y(t,x)\leq \max_{\substack{z\in \partial \Omega \\ 0<s\leq t}}(cs)^{-n/2}e^{-c|z-y|^2 \slash s}\leq \max_{0<s\leq t}(cs)^{-n/2}e^{-c\delta ^2 \slash s},\ t>0,\ x \in \Omega,
$$
by \eqref{a6}. Now the desired result follows readily from this and \eqref{a7} upon noticing that $s \mapsto (cs)^{-n/2}e^{-c\delta ^2 \slash s}$ is  non-decreasing on $(0,2c\delta ^2 \slash n)$.
\end{proof}

\begin{remark}
\label{rm-heatker}
a) The functions $K(t,\cdot,\cdot)$ and $\Gamma(t,\cdot,\cdot)$ being symmetric for all $t>0$, the statement of Lemma \ref{lm-heatker} remains valid for $x \in \Omega _\delta$ and $y\in \Omega$ as well.\\
b) A result similar to Lemma \ref{lm-heatker} can be found in \cite{Mi} for the Dirichlet Laplacian, which corresponds to the operator $A$ in the peculiar case where ${\bf a}$ is the identity matrix. This claim, which was actually first proved by H. Weyl in \cite{We}, is a cornerstone in the derivation of the classical Weyl's asymptotic formula for the eigenvalues counting function (see e.g. \cite{Dod}).\\
c) We refer to \cite{Co} for an alternative proof of Lemma \ref{lm-heatker} that is based on the classical Feynman-Kac formula (see e.g. \cite{SV}) instead of the maximum principle.
\end{remark}

Let us extend $V \in C_0^\infty(\Omega)$ to $\R^n$ by setting $V(x)=0$ for all $x \in \R^n \setminus \Omega$, and,
with reference to \eqref{a3}-\eqref{a3b}, put 
\bel{a9}
\Gamma _0^V(t,x,y)=\Gamma(t,x,y)\ \mbox{and}\
\Gamma^V_{j+1}(t,x,y)=-\int_0^t\int_{\mathbb{R}^n} \Gamma (t-s,x,z) V(z) \Gamma^V_j(s,z,y)dsdz, j \in \N,
\ee
for all $t>0$ and $x,y \in \R^n$. Armed with Lemma \ref{lm-heatker} we may now relate the asymptotic behavior of $A_j^V(t)$ as $t \downarrow 0$ to the one of 
\bel{a10}
B_j^V(t)=\int_{\Omega} \Gamma^V_j(t,x,x)dx,\ t>0,\ j \in \N.
\ee

\begin{proposition}
\label{pr-asy}
Let $j \in \N^*$. Then for each $k \in \N$ we have $A_j^V(t)=B_j^V(t)+O(t^k)$ as $t\downarrow 0$.
\end{proposition}

\begin{proof}
Choose $\delta >0$ so small that $\mbox{supp}(V)\subset \Omega_\delta$, where $\Omega_\delta$ is the same as in Lemma \ref{lm-heatker}, and pick 
$t \in (0,2c\delta^2 \slash n)$. Then, for all $x,y \in \Omega$, we have 
\beas
|\Gamma _1^V(t,x,y)-K_1^V(t,x,y)|& \leq & \int_0^t \int_{\Omega_\delta} \Gamma (t-s,x,z)|V(z)|[\Gamma (s,z,y)-K(s,z,y)]dzds
\\
& + & \int_0^t\int_{\Omega_\delta}  [\Gamma (t-s,x,z)-K(t-s,x,z)]|V(z)|K(s,z,y)dzds,
\eeas
by \eqref{a3b} and \eqref{a9}. This, together with Lemma \ref{lm-heatker} and part a) in Remark \ref{rm-heatker}, yields 
\bel{a11}
|\Gamma _1^V(t,x,y)-K_1^V(t,x,y)|\leq \|V\|_\infty (ct)^{-n/2}e^{-c\delta ^2 \slash t}\left(\int_0^t\int_{\mathbb{R}^n}\Gamma (s,x,z)dz +\int_0^t\int_{\mathbb{R}^n}\Gamma (s,z,y)dz \right),
\ee
for all $t>0$ and a.e. $x, y \in \Omega$. Here we used the estimae $0 \leq K \leq \Gamma$ and the fact that the function $s \mapsto (cs)^{-n/2}e^{-c\delta ^2 \slash s}$ is non-decreasing on $(-\infty,2c\delta ^2 \slash n]$.
Further, due to \eqref{a6}, there is a positive constant $C$, independent of $t$, such that
$$ \int_0^t\int_{\mathbb{R}^n}\Gamma (s,x,z)dz +\int_0^t\int_{\mathbb{R}^n}\Gamma (s,z,y)dz \leq C t,\ t>0, x,y \in \Omega, $$
so we obtain
$$
|\Gamma _1^V(t,x,y)-K_1^V(t,x,y)|\leq (2C\|V\|_\infty) t (ct)^{-n/2}e^{-c\delta ^2 \slash t},\ t>0,\ x,y \in \Omega,
$$
by \eqref{a11}. Similarly, using \eqref{a4} and arguing as above, we get 
$$
|\Gamma _j^V(t,x,y)-K_j^V(t,x,y)|\leq (2C\|V\|_\infty)^j\frac{t^j}{j!} (ct)^{-n/2}e^{-c\delta ^2 \slash t},\ t >0,\ x,y \in \Omega,
$$
by induction on $j \in \N^*$. Now the result follows from this, \eqref{a8} and \eqref{a10}.
\end{proof}

\subsection{The case of a homogeneous metric with degree $0$}
We now express the function $(t,x) \in \R_+^* \times \R^n \mapsto \Gamma_j^V(t,x,x)$, $j \in \N^*$, defined by \eqref{a9}, in terms of the heat kernel $\Gamma$ and the 
perturbation $V$, in the particular case where ${\bf a}$ is homogeneous of degree $0$. The result is as follows.

\begin{lemma}
\label{lemma2.2}
Assume that $\mathbf{a}$ is homogeneous of degree 0. Then for every $j \in \N^*$, $t>0$ and $x \in \R^n$, we have
\beas
\Gamma_j^V(t,x,x) 
& = & (-1)^j t^{j-n/2}
\int_{(\R^n)^n} \int _0^1\int_0^{s_1} \ldots \int_0^{s_{j-1}} \left[ \prod_{i=1}^j \Gamma (s_{i-1}-s_i,x+w_{i-1},x+w_i) V(x+\sqrt{t} w_i) \right] \\
& & \hspace*{5.5cm} \times \Gamma (s_j,x+w_j,x) ds^j dw^j,
\eeas
with $s_0=1$, $w_0=0$, and $d\beta^j=d\beta_1 \ldots d\beta_j$ for $\beta=s,w$.
\end{lemma}

\begin{proof}
The main benefit of dealing with a homogeneous function $\mathbf{a}$ of degree $0$ is the following property:
$$
\Gamma (ts,x,y)=t^{-n/2}\Gamma \left( s,\frac{x}{\sqrt{t}},\frac{y}{\sqrt{t}} \right),\ t, s >0,\ x,y \in \R^n.
$$
From this and the following the identity arising from \eqref{a9} for all $t>0$ and $x,y \in \R^n$,
\beas
\Gamma_j^V(t,x,y) & = & (-1)^jt^j\int_{(\mathbb{R}^n)^2} \int_0^1\int_0^{s_1}\ldots \int_0^{s_{j-1}} \left[ \prod_{i=1}^j\Gamma (t(s_{i-1}-s_i),z_{i-1},z_i)  V(z_i) \right] \Gamma (ts_j,z_j,y) dz^j ds^j,
\eeas
with $z_0=x$,  then follows that
\beas
& & \Gamma_j^V(t,x,y) \\
& = & (-1)^j t^{j-(j+1) n \slash 2} \int_{(\mathbb{R}^n)^n} \int _0^1\int_0^{s_1}\ldots \int_0^{s_{j-1}} \left[ \prod_{i=1}^j \Gamma \left(s_{i-1}-s_i,\frac{z_{i-1}}{\sqrt{t}},\frac{z_i}{\sqrt{t}} \right) V(z_i) \right] \Gamma \left(s_j,\frac{z_j}{\sqrt{t}},\frac{y}{\sqrt{t}} \right) 
dz^j ds^j.
\eeas
Thus, by performing the change of variables $(z_1,\ldots z_j)=\sqrt{t}(w_1,\ldots w_j)+ (x,\ldots ,x)$ in the above integral, we find out that
\beas
\Gamma_j^V(t,x,y)
& = & (-1)^j t^{j-n \slash 2}
\int_{(\R^n)^n} \int _0^1 \int_0^{s_1}\ldots \int_0^{s_{j-1}} 
\left[ \prod_{i=1}^j \Gamma \left(s_{i-1}-s_i,\frac{x}{\sqrt{t}}+w_{i-1},\frac{x}{\sqrt{t}}+w_i \right)V \left(x+\sqrt{t}w_i \right) \right] \nonumber \\
& & \hspace*{5.5cm} \times \Gamma \left(s_j,\frac{x}{\sqrt{t}}+w_j,\frac{y}{\sqrt{t}}\right) dw^j ds^j. 
\eeas
Finally, we obtain the desired result upon taking $y=x$ in the above identity and recalling that $\Gamma=\Gamma_{\mathbf{a}}$ verifies
$$
\Gamma_{\mathbf{a}} \left( t,\frac{x}{\sqrt{t}}+z,\frac{x}{\sqrt{t}}+w \right)=\Gamma_{\mathbf{a}\left(\cdot\, -\frac{x}{\sqrt{t}}\right)}(t,z,w)
= \Gamma_{\mathbf{a}(\sqrt{t}\, \cdot\, -x)}(t,z,w)
=\Gamma_{\mathbf{a}(\sqrt{t}\, \cdot\, )}(t,z+x,w+x)
=\Gamma_{\mathbf{a}}(t,z+x,w+x),
$$
for all $t>0$ and $x,z,w$ in $\R^n$.
\end{proof}

If $\mathbf{a}$ is the identity matrix $\mathbf{I}$, then $\Gamma(t,x,y)$ is explicitly known and coincides with the following Gaussian kernel
\bel{a13}
G(t,x-y)=(4\pi t)^{-n/2}e^{-|x-y|^2 \slash (4t)},\ t>0,\ x,y \in \R^n.
\ee
This and Lemma \ref{lemma2.2} entails the following:

\begin{lemma}
\label{lm-id}
Assume that $\mathbf{a}=\mathbf{I}$. Then, using the same notations as in Lemma \ref{lemma2.2}, we have
\beas
\Gamma_j^V(t,x,x)&=& (-1)^j t^{j -n \slash 2}
\int_{(\R^n)^n} \int_0^1\int_0^{s_1}\ldots \int_0^{s_{j-1}} \left[ \prod_{i=1}^j G(s_{i-1}-s_i,w_{i-1}-w_i)V(x+\sqrt{t}w_i) \right] \\
& & \hspace*{5.6cm} \times G (s_j,w_j)dw^jds^j,
\eeas
for all $t>0$ and $x \in \R^n$, where $G$ is defined by \eqref{a13}.
\end{lemma}


\section{Asymptotic expansion formulae}
\label{sec-af}

In this section we establish the asymptotic expansion formulae \eqref{i3}-\eqref{i4}. The strategy of the proof is, first, to establish \eqref{i3}-\eqref{i4} where 
\bel{e0}
Z^V(t) = \mathrm{tr}(e^{-t H_V}-e^{-t H}),\ t>0,
\ee
is substituted for $Z_\Omega^V(t)$, and, second, to relate the asymptotics of $Z_\Omega^V(t)$ as $t \downarrow 0$ to the one of $Z^V(t)$.

Here $H$ is the selfadjoint operator generated in $L^2(\R^n)$ by the closed quadratic form
$$ \mathfrak{h}[u]=\int_{\R^n} a(x) | \nabla u(x) |^2 dx,\ u \in D(\mathfrak{h})=H^1(\R^n), $$
and $H_V=H+V$ as a sum in the sense of quadratic forms. It is easy to check that $H$ acts on its domain $D(H)=H^2(\R^n)$, the second-order Sobolev space on $\R^n$, as the rhs of \eqref{i1}. Moreover we have $D(H_V)=D(H)$ since $V \in L^{\infty}(\R^n)$. In other words $H$ (resp., $H_V$) may be seen as the extension of the operator $A$ (resp., $A_V$) acting in 
$L^2(\R^n)$, and, due to \eqref{a9} and \eqref{e0}, we have
\bel{e0b}
Z^V(t)=\sum_{j\geq 1} H_j^V(t),\ t>0,\ \mbox{where}\ H_j(t)=\int_{\mathbb{R}^n}\Gamma_j^V(t,x,x)dx,\ j \in \N.
\ee
In light of this and Lemma \ref{lemma2.2}, we apply Taylor's formula to $V\in C_0^\infty (\Omega )$, getting for all $j \geq 1$ and $p \geq 1$,
\bel{e1}
\prod_{k=1}^j V(x+tw_k)=\sum_{\ell =0}^{p-1}t^\ell \left[ \sum_{|\alpha _1|+\ldots |\alpha _j|=\ell}\frac{1}{\alpha _1!\ldots \alpha _j!}\prod_{k=1}^j\partial ^{\alpha _k}V(x)w_k^{\alpha _k}\right]+t^pR_j^p (t,x,w_1\ldots ,w_j),
\ee
where
\bel{e2}
R_j^p (t,x,w_1\ldots ,w_j)=\sum_{|\alpha _1|+\ldots |\alpha _j|=p}\frac{p}{\alpha _1!\ldots \alpha _j!}\int_0^1(1-s)^{p-1}\prod_{k=1}^j\partial ^{\alpha _k}V(x+stw_k)w_k^{\alpha _k}ds.
\ee
For the sake of notational simplicity we note 
\bel{e2b}
\alpha ^j=(\alpha _1^j,\ldots, \alpha _j^j)\in (\mathbb{N}^n)^j,\ \alpha ^j!=\prod_{k=1}^j \alpha_k^j!\ \mbox{and}\ W_j ^{\alpha ^j}= \prod_{k=1}^j w_k^{\alpha_k^j}, 
\ee
so that \eqref{e1}-\eqref{e2} may be reformulated as
\bel{e3}
\prod_{k=1}^j V(x+tw_k)=\sum_{\ell =0}^{p-1}t^\ell \left[ \sum_{|\alpha ^j|=\ell}\frac{W_j ^{\alpha ^j}}{\alpha ^j!}\prod_{k=1}^j\partial ^{\alpha _k^j}V(x)\right]+t^pR_j^p (t,x,w_1\ldots ,w_j),
\ee
with
\bel{e4}
R_j^p (t,x,w_1\ldots ,w_j)=\sum_{ |\alpha ^j|=p}\frac{pW_j ^{\alpha ^j}}{ \alpha ^j!}\int_0^1(1-s)^{p-1}\prod_{k=1}^j\partial ^{\alpha _k}V(x+stw_k)ds,\ j,p \in \N^*.
\ee
Next, with reference to \eqref{e2b} we define for further use 
\bel{e5}
c_{\alpha ^j}(x)=
\frac{1}{\alpha ^j!}\int_{(\mathbb{R}^n)^n} \int _0^1\int_0^{s_1}\ldots \int_0^{s_{j-1}} W_j ^{\alpha ^j} \left[ \prod_{i=1}^j \Gamma (s_{i-1}-s_i,x+w_{i-1},x+w_i) \right] 
\Gamma (s_j,x+w_j,x) dw^j ds^j,
\ee
where, as usual, $s_0=1$, $w_0=0$, and $d u^j$ stands for $d u_1 \ldots d u_j$ with $u=s, w$, and we put
\bel{e6}
\mathscr{P}_{\alpha ^j}(V)=\int_\Omega c_{\alpha ^j}(x)\prod_{k=1}^j\partial ^{\alpha _k^j}V(x)dx,\ j \in \N^*.
\ee
We now state the main result of this section.
\begin{proposition}
\label{pr-asymptotic}
Let $p \in \N \setminus \{ 0, 1, 2 \}$. Then, under the assumption \eqref{i0}, $Z^V(t)$ and $Z_\Omega^V(t)$ have the following asymptotic expansion
$$ \sum_{\ell =2}^{p-1}t^{\ell /2}\mathscr{P}_{\ell }(V) +O(t^{p/2})\ \mbox{as}\ t \downarrow 0, $$
where
$$
\mathscr{P}_\ell (V)=\sum_{1\leq j\leq \ell /2}(-1)^j \sum_{|\alpha ^j|=\ell-2j}\mathscr{P}_{\alpha ^j}(V),
$$
the coefficients $\mathscr{P}_{\alpha ^j}(V)$ being given by \eqref{e5}-\eqref{e6}.
\end{proposition}

\begin{proof}
In view of \eqref{a10} and Lemma \ref{lemma2.2}, we have
$$
t^nB_j ^V(t^2)= (-1)^j \sum_{\ell =0}^{p-1}t^{\ell +2j}  \sum_{|\alpha ^j|=\ell}\mathscr{P}_{\alpha ^j}(V) +O(t^{p+2j}),\ t>0,\ j \in \N^*,
$$
hence
$$
t^nB_j ^V(t^2)= (-1)^j \sum_{\ell =2j}^{p-1} t^{\ell }  \sum_{|\alpha ^j|=\ell-2j}\mathscr{P}_{\alpha ^j}(V) +O(t^p),\ t>0,\ j \in \N^*.
$$
Summing up the above identity over all integers $j$ between $1$ and $(p-1) \slash 2$,
we find that
\beas
t^n \sum_{1 \leq j \leq (p-1) \slash 2} B_j ^V(t^2) &= & \sum_{1 \leq j \leq (p-1) \slash 2} (-1)^j \sum_{\ell =2j}^{p-1}t^{\ell }  \sum_{|\alpha ^j|=\ell-2j}\mathscr{P}_{\alpha ^j}(V) +O(t^p) \nonumber \\
& = & \sum_{\ell=2}^{p-1} t^{\ell } \sum_{1 \leq j \leq  \ell \slash 2} (-1)^j \sum_{|\alpha ^j|=\ell-2j}\mathscr{P}_{\alpha ^j}(V) +O(t^p).
\eeas
As a consequence we have $t^n \sum_{1 \leq j \leq (p-1) \slash 2} B_j ^V(t^2) = \sum_{\ell =2}^{p-1} t^{\ell } \mathscr{P}_\ell (V) +O(t^p)$, hence
\bel{e9}
t^n \sum_{j \geq 1} B_j ^V(t^2) = \sum_{\ell =2}^{p-1} t^{\ell } \mathscr{P}_\ell (V) +O(t^p).
\ee
Next, bearing in mind that $V$ is supported in $\Omega$, we see that $\mathscr{P}_{\alpha ^j}(V)=\int_{\R^n} c_{\alpha ^j}(x)\prod_{k=1}^j\partial ^{\alpha _k^j}V(x)dx$ for all $j \in \N^*$.
This entails
\bel{e10}
t^n \sum_{j \geq 1} H_j ^V(t^2) = \sum_{\ell =2}^{p-1} t^{\ell } \mathscr{P}_\ell (V) +O(t^p),
\ee
upon substituting \eqref{e0b} (resp., $H_j^V$) for \eqref{a10} (resp., $B_j^V$) in the above reasonning. Finally, putting \eqref{a8c}, \eqref{e9} and Proposition \ref{pr-asy}
(resp., \eqref{e0b} and \eqref{e10}) together we obtain the result for $Z_\Omega^V$ (resp., $Z^V$).
\end{proof}
Proposition \ref{pr-asymptotic} immediately entails the:
\begin{corollary}
\label{cor-asymptotic}
Let  $V_0\in C_0^\infty (\Omega )$. Then, under the conditions of Proposition \ref{pr-asymptotic}, each $V\in \mbox{Is}(V_0)$ verifies
$$\mathscr{P}_{\ell }(V)=\mathscr{P}_{\ell }(V_0),\ \ell \geq 2. $$
\end{corollary}



In the particular case where $\mathbf{a}=\mathbf{I}$, \eqref{e5} may be rewritten as
\bel{e11}
c_{\alpha ^j}(x)=c_{\alpha ^j}=
\frac{1}{\alpha ^j!}\int_{(\R^n)^n} \int _0^1 \int_0^{s_1} \ldots \int_0^{s_{j-1}} W_j^{\alpha^j} G(1-s_1,w_1) \prod_{k=1}^j G(s_k-s_{k+1},w_k-w_{k+1}) dw^j ds^j,
\ee
with $s_{j+1}=w_{j+1}=0$. Here $G$ is defined by \eqref{e2b} and the notations $\alpha^j!$ and $W_j^{\alpha^j}$ are the same as in \eqref{e2b}.
Thus we have
\bel{e12}
\mathscr{P}_{\alpha^j}(V)=c_{\alpha ^j} P_{\alpha^j}(V),\ P_{\alpha^j}(V)=\int_\Omega \prod_{k=1}^j\partial ^{\alpha _k^j}V(x)dx,
\ee
from \eqref{e6}, hence Proposition \ref{pr-asymptotic} entails the:
\begin{proposition}
\label{pr-asymptotic2}
Assume that $\mathbf{a}=\mathbf{I}$. Then, for any  $p \in \N^*$, the asymptotics of $Z^V(t)$ and $Z_\Omega(t)$ as $t \downarrow 0$ have the expression
$$ \sum_{\ell =1}^{p}t^{\ell }\mathcal{P}_{2\ell }(V) +O(t^{p+1}), $$
where
\bel{e12b}
\mathcal{P}_\ell (V)=\sum_{1\leq j\leq \ell /2}(-1)^j \sum_{|\alpha ^j|=\ell-2j} \mathscr{P}_{\alpha_ j}(V)=\sum_{1\leq j\leq \ell /2}(-1)^j \sum_{|\alpha ^j|=\ell-2j} c_{\alpha ^j} P_{\alpha^j}(V),
\ee
the coefficients $\mathscr{P}_{\alpha_ j}(V)$, $c_{\alpha ^j}$ and $P_{\alpha^j}(V)$ being defined by \eqref{e11}-\eqref{e12}.
\end{proposition}

\begin{proof}
Upon performing the change of variables $(w_1,\ldots ,w_j)\longrightarrow (-w_1,\ldots ,-w_j)$ in the rhs of \eqref{e11} we get that $c_{\alpha ^j}=(-1)^{|\alpha ^j|}c_{\alpha^^j}$.
Therefore $c_{\alpha ^j}=0$ hence $\mathscr{P}_{\alpha_ j}(V)=0$ by \eqref{e12}, for $|\alpha ^j|$ odd. As a consequence we have
$$
\mathcal{P}_{2\ell +1}(V)=\sum_{1 \leq j \leq \ell} (-1)^j \sum_{|\alpha ^j|=2(\ell -j)+1} \mathscr{P}_{\alpha_ j}(V)  =0.
$$
Thus, applying \eqref{e9} where $2(p+1)$ is substituted for $p$, we find out that
$$
t^n \sum_{j \geq 1} B_j ^V(t^2) = \sum_{\ell =2}^{2p+1} t^{\ell } \mathscr{P}_\ell (V) +O(t^{2(p+1)})=\sum_{\ell =1}^{p} t^{2 \ell } \mathscr{P}_{2 \ell} (V) +O(t^{2(p+1)}),
$$
which, in turns, yields
$$ 
t^{n/2} \sum_{j \geq 1} B_j ^V(t^2) = \sum_{\ell =1}^{p}t^{\ell }\mathscr{P}_{2\ell }(V) +O(t^{p+1}).
$$
Now the result follows from this by arguing as in the proof of Proposition \ref{pr-asymptotic}.
\end{proof}

\begin{remark}
It is clear that the asymptotic formula stated in Proposition \ref{pr-asymptotic} (resp., Proposition \ref{pr-asymptotic2}) for $Z^V$ remains valid upon substituting $\int_{\mathbb{R}^n} c_{\alpha ^j}(x)\prod_{k=1}^j\partial ^{\alpha _k^j}V(x)dx$ (resp., $c_{\alpha^j} \int_{\mathbb{R}^n} \prod_{k=1}^j\partial ^{\alpha _k^j}V(x)dx$) for $\mathscr{P}_{\alpha^j}(V)$, if $V$ is taken in the Schwartz class $\mathscr{S}(\mathbb{R}^n)$.
\end{remark}


\section{Two parameter integrals}

In this section we collect useful properties of two parameter integrals appearing in the proof of Theorem \ref{thm-comp}, presented in section \ref{sec-proof}. As a preamble we consider the integral
\bel{ap1}
I_n(f)=\int_{\mathbb{R}^n}\int_{\mathbb{R}^n}\int_0^1\int_0^{s_1}f(w_1,w_2)G(1-s_1,w_1)G(s_1-s_2,w_1-w_2)G(s_2,w_2)dw_1dw_2ds_1ds_2,
\ee
where $f \in C^\infty (\mathbb{R}^n\times \mathbb{R}^n)$ and $G$ is defined by \eqref{a13}. 
For all $\sigma \in \mathfrak{\sigma}_n$, the set of permutations of $\{ 1,\ldots ,n\}$, 
and all $z=(z_1,\ldots ,z_n)\in \mathbb{R}^n$, we write $\sigma z=(z_{\sigma(1)},\ldots , z_{\sigma (n)})$. Similarly, for every $w_1, w_2\in \mathbb{R}^n$, we note $\sigma (w_1,w_2)= (\sigma w_1,\sigma w_2)$ and $f\circ \sigma (w_1,w_2)=f(\sigma (w_1,w_2))$. The following result gathers several  properties of $I_n$ that are required in the remaining part of this section.

\begin{lemma}
\label{lm-int}
Let $f \in C^\infty (\mathbb{R}^n\times \mathbb{R}^n)$. Then it holds true that:
\begin{enumerate}[i)]
\item $I_n(f)=I_n(Sf)$, where $S$ denotes the ``mirror symmetry" operator acting as $S f(w_1,w_2)=f(w_2,w_1)$;
\item $I_n(f)=I_n(f \circ \sigma )$ for all $\sigma \in \mathfrak{\sigma}_n$;
\item If there are $f_k \in C^\infty (\R \times \R)$, $k=1,\ldots,n$, such that
$$f(w_1,w_2)=\prod_{k=1}^n f_k(w_1^k, w_2^k),\ w_i=(w_i^1,\ldots,w_i^n), i=1,2, $$
and if any of the $f_k$ is an odd function of $(w_1^k, w_2^k)$, then we have $I_n(f)=0$.
\end{enumerate}
\end{lemma}

\begin{proof}
i) Upon performing successively the two changes of variables $\tau _1=1-s_1$ and $\tau _2=1-s_2$ in the rhs of \eqref{ap1}, we get that
\beas
I_n(f) &=&\int_{\mathbb{R}^n}\int_{\mathbb{R}^n}\int_0^1\int_{\tau _1}^1f(w_1,w_2)G(\tau_1,w_1)G(\tau_2-\tau_1,w_1-w_2)G(1-\tau_2,w_2)dw_1dw_2d\tau_2d\tau_1\\
\\
&=&\int_{\mathbb{R}^n}\int_{\mathbb{R}^n}\int_0^1\int_0^{\tau _2}f(w_1,w_2)G(\tau_1,w_1)G(\tau_2-\tau_1,w_1-w_2)G(1-\tau_2,w_2)dw_1dw_2d\tau_1d\tau_2,
\eeas
so the result follows by relabelling $(w_1,w_2)$ as $(w_2,w_1)$.\\
ii) In light of \eqref{a13} we have $G(t,w)=G(t,\sigma^{-1} w)$ for all $t>0$, $w \in \R^n$ and $\sigma \in \mathfrak{\sigma}_n$, hence
$I_n(f)$ is equal to
$$ \int_{\mathbb{R}^n}\int_{\mathbb{R}^n}\int_0^1\int_0^{s_1}f(w_1,w_2)G(1-s_1,\sigma ^{-1}w_1)G(s_1-s_2,\sigma ^{-1}w_1-\sigma ^{-1}w_2)G(s_2,\sigma ^{-1}w_2)dw_1dw_2ds_1ds_2,$$
according to \eqref{ap1}.
The result follows readily from this upon performing the change of variable $(\widetilde{w}_1,\widetilde{w}_2)=\sigma^{-1}(w_1,w_2)$.\\
iii) This point is a direct consequence of the obvious identity $I_n(f)=\prod_{k=1}^n I_1(f_k)$, arising from \eqref{a13} and \eqref{ap1}.
\end{proof}

We turn now to evaluating integrals of the form
\bel{ap2}
I_{\alpha ,\beta}=I_{\alpha ,\beta}(s_1,s_2)=\int_{\mathbb{R}}\int_{\mathbb{R}}x^\alpha y^\beta g(1-s_1,x)g(s_1-s_2, x-y)g(s_2,y)dxdy,\ \alpha , \beta \in \N,\ s_1 , s_2 \in \R,
\ee
for $\alpha +\beta$ even, where $g$ denotes the one-dimensional Gaussian kernel defined by \eqref{a13} in the particular case where $n=1$.
This can be achieved upon using the following result.
\begin{lemma}
\label{lm-int2}
For all $\alpha, \beta \in \N$ and all $s_1,s_2 \in \R$, we have:
\begin{enumerate}[i)]
\item $I_{1,1}(s_1,s_2)=2(4\pi )^{-1/2}(1-s_1)s_2$; \vspace*{.1cm}
\item $I_{\alpha ,\beta}(s_1,s_2)=2(1-s_1)s_2\left[ 2(\alpha -1)(\beta -1)(s_1-s_2)I_{\alpha -2,\beta -2}(s_1,s_2) +(\alpha +\beta -1)I_{\alpha -1,\beta -1}(s_1,s_2)\right]$; \vspace*{.1cm}
\item $I_{\alpha ,\beta}(s_1,s_2)=2(1-s_1)\left[(\alpha -1)s_1I_{\alpha -2,\beta}(s_1,s_2)+\beta s_2 I_{\alpha -1,\beta -1}(s_1,s_2) \right]$; \vspace*{.1cm}
\item $I_{\alpha ,\beta}(s_1,s_2)=2(1-s_1)\left[(\alpha +\beta -1)s_1I_{\alpha -2,\beta}(s_1,s_2)-2 \beta (\beta -1) s_2(s_1-s_2)I_{\alpha -2,\beta -2}(s_1,s_2) \right]$; \vspace*{.1cm}
\item $I_{2\alpha ,0}(s_1,s_2)=(4\pi )^{-1/2} (2\alpha )! \slash (\alpha !) s_1^\alpha (1-s_1)^\alpha$; \vspace*{.1cm}
\item $I_{0,2\alpha}(s_1,s_2)=(4\pi )^{-1/2} (2\alpha )! \slash (\alpha !) s_2^\alpha (1-s_2)^\alpha$.
\end{enumerate}
\end{lemma}

\begin{proof}
a) In light of the basic identity 
\bel{ap3}
zg(t,z)=-2t \partial_z g(t,z),\ t>0,\ z \in \R, 
\ee
we have
\beas
& & \int_{\mathbb{R}}\int_{\mathbb{R}} xyg(1-s_1,x) g(s_1-s_2, x-y)g(s_2,y)dxdy \\
&= &-2(1-s_1)\int_{\mathbb{R}}\int_{\mathbb{R}}y\partial_xg(1-s_1,x)g(s_1-s_2, x-y)g(s_2,y)dxdy \\
&=& 2(1-s_1)\int_{\mathbb{R}}\int_{\mathbb{R}}yg(1-s_1,x)\partial_xg(s_1-s_2, x-y)g(s_2,y)dxdy,
\eeas
by integrating by parts. Thus, applying \eqref{ap3} once more, we obtain that
\bea
& & \int_{\mathbb{R}}\int_{\mathbb{R}} xyg(1-s_1,x) g(s_1-s_2, x-y)g(s_2,y)dxdy \nonumber\\
&=& -2(1-s_1)\int_{\mathbb{R}}\int_{\mathbb{R}}yg(1-s_1,x)\partial_yg(s_1-s_2, x-y)g(s_2,y)dxdy \nonumber \\
&=& 2(1-s_1)\int_{\mathbb{R}}\int_{\mathbb{R}}yg(1-s_1,x)g(s_1-s_2, x-y)\partial _yg(s_2,y)dxdy +2(1-s_1)(4\pi )^{-1/2} \nonumber  \\
&=& 2(1-s_1)\int_{\mathbb{R}}yg(1-s_2,y)\partial _yg(s_2,y)dy+2(1-s_1)(4\pi )^{-1/2}, \label{ap4}
\eea
with the help of the reproducing property. On the other hand, an integration by parts providing
\beas
\int_{\mathbb{R}}yg(1-s_2,y)\partial _yg(s_2,y)dy&=&  -\int_{\mathbb{R}}g(1-s_2,y)g(s_2,y)dy-\int_{\mathbb{R}}y\partial _yg(1-s_2,y)g(s_2,y)dy \\
&=&-(4\pi)^{-1/2}-\frac{s_2}{1-s_2} \int_{\mathbb{R}}yg(1-s_2,y)\partial _yg(s_2,y)dy,
\eeas
we get that $\int_{\mathbb{R}}yg(1-s_2,y)\partial _yg(s_2,y)dy=-(4\pi)^{-1/2}(1-s_2)$. Thus Part i) follows from this and \eqref{ap4}.\\
b) Applying \eqref{ap3} with $z=x$ and $t=1-s_1$ we find that
\beas
I_{\alpha ,\beta}(s_1,s_2) &=& -2(1-s_1)\int_{\mathbb{R}}\int_{\mathbb{R}}x^{\alpha -1}y^\beta \partial _xg(1-s_1,x)g(s_1-s_2, x-y)g(s_2,y)dxdy \\
&=&2(\alpha -1)(1-s_1)\int_{\mathbb{R}}\int_{\mathbb{R}}x^{\alpha -2}y^\beta g(1-s_1,x)g(s_1-s_2, x-y)g(s_2,y)dxdy \\
& & -\frac{2(1-s_1)}{2(s_1-s_2)}\int_{\mathbb{R}}\int_{\mathbb{R}}x^{\alpha -1}y^\beta (x-y)g(1-s_1,x)g(s_1-s_2, x-y)g(s_2,y)dxdy,
\eeas
by integrating by parts wrt $x$, so we get
\bel{ap5}
(1-s_2) I_{\alpha ,\beta}(s_1,s_2)=2(\alpha -1)(1-s_1)(s_1-s_2)I_{\alpha -2,\beta}(s_1,s_2)+(1-s_1)I_{\alpha -1,\beta +1}(s_1,s_2).
\ee
Doing the same with $z=y$ and $t=s_2$ we obtain that
\bel{ap6}
s_1 I_{\alpha ,\beta}(s_1,s_2)=2 (\beta -1) (s_1-s_2) s_2 I_{\alpha ,\beta -2}(s_1,s_2)+s_2 I_{\alpha +1,\beta -1}(s_1,s_2).
\ee
Thus, upon successively substituting $(\alpha -1,\beta +1)$ and $(\alpha -2,\beta )$ for $(\alpha,\beta)$ in \eqref{ap6}, we find that
\bel{ap7}
s_1 I_{\alpha -1 ,\beta +1}(s_1,s_2)=2 \beta s_2(s_1-s_2) I_{\alpha -1,\beta -1}(s_1,s_2)+s_2 I_{\alpha ,\beta}(s_1,s_2)
\ee
and
\bel{ap8}
s_1 I_{\alpha -2 ,\beta}(s_1,s_2)=2 (\beta -1) (s_1-s_2) I_{\alpha -2,\beta -2}(s_1,s_2)+s_2 I_{\alpha -1,\beta -1}(s_1,s_2).
\ee
Plugging \eqref{ap7}-\eqref{ap8} in \eqref{ap5} we end up getting Part ii). Further we obtain Part iii) by following the same lines as in the derivation of Part ii), and Part iv) is a direct consequence of Parts ii) and iii).\\
c) Arguing as in the derivation of Part i) in a), we establish for any $\alpha \geq 2$ that
$$
I_{\alpha ,0}(s_1,s_2)=2(\alpha -1) s_1 (1-s_1) I_{\alpha -2,0}(s_1,s_2).
$$
This and the obvious identity $I_{0,0}(s_1,s_2)=(4\pi )^{-1/2}$ yields Part v) upon proceeding by induction on $\alpha$. Finally, Part vi) follows from 
Part v) upon noticing from \eqref{ap2} that $I_{0,\alpha}(s_1,s_2) = I_{\alpha,0}(1-s_2,1-s_1)$.
\end{proof}

Further, for all $\alpha =(\alpha_k)_{1 \leq k \leq n}$ and $\beta =(\beta _k)_{1 \leq k \leq n}$ in $\mathbb{N}^n$, we put
\bel{ap9}
\mathscr{I}(\alpha ,\beta )=\int_0^1\int_0^{s_1} \left( \int_{\mathbb{R}^n}\int_{\mathbb{R}^n}x^\alpha y^\beta G(1-s_1,x)G(s_1-s_2, x-y)G(s_2,y)dxdy \right) ds_1ds_2,
\ee
and establish the:
\begin{lemma}
\label{lm-I}
For each $\alpha =(\alpha_k)_{1 \leq k \leq n}$ and $\beta =(\beta _k)_{1 \leq k \leq n}$ in $\mathbb{N}^n$ we have:
\begin{enumerate}[i)]
\item 
$\mathscr{I}(\alpha ,\beta )=\int_0^1\int_0^{s_1}\prod_{k=1}^n I_{\alpha_k,\beta _k}(s_1,s_2)ds_1ds_2$.
\item
$\mathscr{J}(\alpha ,\beta )=\mathscr{J}(\beta ,\alpha )$.
\item 
$\mathscr{J}(\alpha ,\beta )=0$ if any of sums $\alpha _k+\beta _k$ for $1 \leq k \leq n$, is odd.
\end{enumerate}
\end{lemma}

\begin{proof}
Part i) follows readily from the identity $G(s,z)=\prod_{k=1}^n g(s,z_k)$ arising from \eqref{a13} for all $s \in \R^*$ and all $z=(z_k)_{1 \leq k \leq n} \in \R^n$, and from the very definitions \eqref{ap2} and \eqref{ap9}. Next, Part ii) is a direct consequence first assertion of Lemma \eqref{lm-int}, while Part iii) follows from the third point of Lemma \eqref{lm-int}.
\end{proof}


\section{Proof of Theorem \ref{thm-comp}}
\label{sec-proof}

We start by establishing two identities which are useful for the proof of Theorem \ref{thm-comp}.

\subsection{Two useful identities}
They are collected in the following:
\begin{proposition}
\label{pr-id}
Let $V \in C_0^\infty(\Omega)$ be real-valued and assume that ${\bf a}={\bf I}$. Then, with reference to the definitions \eqref{e11}-\eqref{e12}, we have
\bel{pf1}
(4\pi )^{n/2}\sum_{|\alpha ^2|=2}c_{\alpha ^2}P_{\alpha ^2}(V )= -\frac{1}{12}\int_\Omega |\nabla V|^2dx 
\ee
and
\bel{pf2}
(4\pi)^{n/2}\sum_{|\alpha ^2|=4}c_{\alpha ^2}P_{\alpha ^2}(V)=\frac{1}{120}\sum_{k}\int_\Omega \left( \partial _{kk}^2 V\right)^2dx+\frac{13}{360}\sum_{k\neq \ell }\int_\Omega \left( \partial _{k\ell}^2 V\right)^2dx.
\ee
\end{proposition}

\begin{proof}
Since
\bel{pf2b}
c_{\alpha ^2}=\frac{\mathscr{I}(\alpha _1^2,\alpha _2^2)}{\alpha^2!},\ \alpha^2=(\alpha_1^2,\alpha_2^2),
\ee
by \eqref{e11} and \eqref{ap9}, we know from the two last points in Lemma \ref{lm-I} that
\bel{pf3}
c_{\alpha ^2}=0\ \mbox{if the sum}\ (\alpha _1^2)_k+(\alpha _2^2)_k\ \mbox{is odd for any}\ k \in \{1,\ldots,n\},
\ee
and
\bel{pf4}
c_{\alpha ^2}=c_{\widetilde{\alpha} ^2}\ \mbox{for}\ \widetilde{\alpha}^2=(\alpha ^2_2,\alpha ^2_1),
\ee
We first compute $\sum_{|\alpha ^2|=4}c_{\alpha ^2}P_{\alpha ^2}(V )$. In what follows we note $(0,\ldots ,\underset{k}{\beta},\ldots ,0)$, $1 \leq k \leq n$, $\beta \in \R$, the vector $(\beta_j)_{1 \leq j \leq n} \in \R^n$ such that $\beta_j=0$ for all $1 \leq j \neq k \leq n$ and $\beta_k = \beta$. In view of \eqref{pf2b} we apply the first point in Lemma \ref{lm-I} for $\alpha ^2=((0,\ldots \underset{k}{2}, \ldots 0),(0,\ldots ,0))$, $1 \leq k \leq n$,
getting
\bel{pf5}
c_{\alpha ^2}=\int_0^1\int_0^{s_1} I_{0,0}(s_1,s_2)^{n-1} I_{2,0}(s_1,s_2) ds_1ds_2=\frac{(4\pi )^{-(n-1)/2}}{2}\int_0^1\int_0^{s_1}I_{2,0}(s_1,s_2) ds_1ds_2
=\frac{(4\pi )^{-n/2}}{12},
\ee
with the aid of Part v) in Lemma \ref{lm-int2}. Similarly, for $\alpha ^2=((0,\ldots ,\underset{k}{1},\ldots ,0),(0,\ldots ,\underset{k}{1},\ldots ,0))$, $1 \leq k \leq n$, we use the first part of Lemma \ref{lm-int2} and obtain that
\bel{pf6}
c_{\alpha ^2}=(4\pi )^{-(n-1)/2}\int_0^1\int_0^{s_1} I_{1,1}(s_1,s_2) ds_1ds_2
=\frac{(4\pi )^{-n/2}}{12}.
\ee
In light of \eqref{pf3}-\eqref{pf4} we deduce from \eqref{pf5}-\eqref{pf6} that
$$ \sum_{|\alpha ^2|=2}c_{\alpha ^2}P_{\alpha ^2}(V )= \frac{1}{6}(4\pi)^{-n/2}\int_\Omega \Delta V V dx + \frac{1}{12}(4\pi)^{-n/2}\int_\Omega |\nabla V|^2dx. $$
Taking into account that $\int_\Omega \Delta V V dx=-\int_\Omega |\nabla V|^2dx$, we obtain \eqref{pf1} from the above line.

We now compute $\sum_{|\alpha ^2|=4}c_{\alpha ^2}P_{\alpha ^2}(V )$. As a preamble we first invoke Lemma \ref{lm-int2} and get simultaneously
\bel{pf8}
I_{2,2}(s_1,s_2)=2(1-s_1)[s_1 I_{0,2}(s_1,s_2)+2s_2I_{1,1}(s_1,s_2)]
=4(4\pi )^{-1/2}(1-s_1)s_2[s_1(1-s_2)+2(1-s_1)s_2],
\ee
and
\bel{pf9}
I_{3,1}(s_1,s_2)=2(1-s_1)[2s_1 I_{1,1}(s_1,s_2)+s_2I_{2,0}(s_1,s_2)]
=12(4\pi )^{-1/2}(1-s_1)^2s_1s_2,
\ee
from Part iii), and
\bel{pf10}
I_{4,0}(s_1,s_2)=12(4\pi )^{-1/2}s_1^2(1-s_1)^2,
\ee
from Part v). Thus, for all $k \in \{1,\ldots,n\}$ it follows from fhe first part of Lemma \ref{lm-I} and \eqref{pf10} upon taking $\alpha ^2=((0,\ldots ,\underset{k}{4},\ldots ,0),(0,\ldots ,0))$ in \eqref{pf2b} that 
\bel{pf11}
c_{\alpha ^2}=\frac{(4\pi)^{-(n-1)/2}}{4!}\int_0^1\int_0^{s_1}I_{4,0}(s_1,s_2)ds_1ds_2
= \frac{1}{120}(4\pi)^{-n/2}.
\ee
Further, choosing $\alpha ^2=((0,\ldots ,\underset{k}{3},\ldots ,0),(0,\ldots ,\underset{k}{1},\ldots ,0))$ we deduce in the same way from \eqref{pf9} that,
\bel{pf12}
c_{\alpha ^2}=\frac{(4\pi)^{-(n-1)/2}}{3!}\int_0^1\int_0^{s_1}I_{3,1}(s_1,s_2)ds_1ds_2=\frac{1}{60}(4\pi)^{-n/2},
\ee
and with $\alpha ^2=((0,\ldots ,\underset{k}{2},\ldots ,0),(0,\ldots ,\underset{k}{2},\ldots ,0))$, we get from \eqref{pf8} that
\bel{pf13}
c_{\alpha ^2}= \frac{(4\pi)^{-(n-1)/2}}{2!2!}\int_0^1\int_0^{s_1}I_{2,2}(s_1,s_2)ds_1ds_2
= \frac{1}{40}(4\pi)^{-n/2}.
\ee
Finally, upon taking $\alpha^2=((0,\ldots ,\underset{k}{2},\ldots ,0),(0,\ldots ,\underset{\ell }{2},\ldots ,0))$ in \eqref{pf2b}, for $1 \leq k \neq \ell \leq n$, we derive from the two last parts of Lemma \ref{lm-int2} that
\bel{pf14}
c_{\alpha ^2}= \frac{(4\pi)^{-(n-2)/2}}{2!2!}\int_0^1\int_0^{s_1}I_{2,0}(s_1,s_2) I_{0,2}(s_1,s_2)ds_1ds_2
= \frac{1}{72}(4\pi)^{-n/2},
\ee
while the choice $\alpha ^2=((0,\ldots ,\underset{k}{1},\ldots ,\underset{\ell}{1},\ldots 0),(0,\ldots ,\underset{k}{1},\ldots ,\underset{\ell }{1},\ldots ,0))$ leads to
\bel{pf15}
c_{\alpha ^2}= (4\pi)^{-(n-2)/2}\int_0^1\int_0^{s_1}I_{1,1}(s_1,s_2)^2ds_1ds_2
= \frac{1}{45}(4\pi)^{-n/2},
\ee
with the aid of the first part. Putting \eqref{pf11}--\eqref{pf15} together and recalling \eqref{pf3}-\eqref{pf4} we end up getting \eqref{pf2}.
\end{proof}
Armed with Proposition \ref{pr-id} we are now in position to prove Theorem \ref{thm-comp}.

\subsection{Completion of the proof}
\label{sec-coe}
By applying the reproducing property \eqref{a5} to the kernel $G$, defined in \eqref{a13}, we derive from \eqref{e11} for all $j \geq 1$ that
\bel{pf16}
c_{\alpha^j=0}= \int_{(\mathbb{R}^n)^n} \int _0^1\int_0^{s_1}\ldots \int_0^{s_{j-1}}G(1-s_1,w_1) \prod_{k=1}^j G(s_k-s_{k+1},w_k-w_{k+1}) dw^j ds^j=\frac{(4\pi )^{-n/2}}{j!},
\ee
where $s_{j+1}=w_{j+1}=0$.
In light of \eqref{e12}-\eqref{e12b}, \eqref{pf16} then yields that
\bel{pf17}
\mathcal{P}_2(V)=-c_{\alpha^1=0} P_{\alpha^1=0}(V)=-(4\pi)^{-n/2} \int_\Omega Vdx.
\ee
Next, bearing in mind that the potential $V$ is compactly supported in $\Omega$, we notice from \eqref{e12} that
\bel{pf17b}
P_{\alpha^1}(V) = \int_{\Omega} \partial^{\alpha^1} V(x) dx = 0,\ | \alpha^1 | \geq 1.
\ee
As a consequence we have
\bel{pf18}
\mathcal{P}_4(V)=c_{\alpha^2=0} P_{\alpha^2=0}(V)-\sum_{| \alpha^1 | =2} c_{\alpha^1} P_{\alpha^1}(V)=\frac{(4\pi)^{-n/2}}{2}\int_\Omega V(x)^2 dx.
\ee
Further, as
$\mathcal{P}_6=-c_{\alpha^3=0} P_{\alpha^3=0}(V)+\sum_{| \alpha^2 | =2} c_{\alpha^2} P_{\alpha^2}(V)-\sum_{| \alpha^1 | =4} c_{\alpha^1} P_{\alpha^1}(V)$,
it follows from \eqref{pf1} and \eqref{pf16} that
\bel{pf19}
\mathcal{P}_6(V)=-\frac{(4\pi)^{-n/2}}{6}\left(\frac{1}{2}\int_\Omega |\nabla V(x) |^2dx+\int_\Omega V(x)^2 dx\right).
\ee
Finally, since $\int_\Omega \partial_{k m}^2V (x) V(x)^2dx=-2\int_\Omega \partial_k V(x) \partial_m V(x) V(x) dx$ for all natural numbers $1 \leq k, m \leq n$, by integrating by parts, we see that there is a constant $C_n$ depending only on $n$ such that we have
$$
\left| \sum_{|\alpha^3|=2} c_{\alpha ^3} P_{\alpha^3}(V) \right| \leq C_n \|V\|_\infty \int_\Omega |\nabla V(x)|^2dx,
$$
according to \eqref{e12}. This, together with the identity
$$
\mathcal{P}_{8} (V)=c_{\alpha ^4 =0}P_{\alpha ^4 =0}(V) - \sum_{|\alpha ^3|=2}c_{\alpha ^3}P_{\alpha ^3}(V) +\sum_{|\alpha ^2|=4}c_{\alpha ^2}P_{\alpha ^2}(V )-\sum_{|\alpha ^1|=6}c_{\alpha ^1}P_{\alpha^1}(V),
$$
arising from \eqref{e12b}, and \eqref{pf2}, \eqref{pf16}, \eqref{pf17b}, then yield
\bel{pf20}
\sum_{|\gamma |=2}\int_\Omega |\partial^\gamma V(x)|^2dx+\int_\Omega V(x)^4dx\leq C_n'\left(|\mathcal{P}_8(V)|+\|V\|_\infty \int_\Omega |\nabla V(x)|^2dx.\right),
\ee
for some constant $C_n'>0$ depending only on $n$. In light of \eqref{pf19}-\eqref{pf20} the set $\mbox{Is}(V_0)\cap \mathcal{B}$ is thus bounded in $H^2(\Omega )$ from
Corollary \ref{cor-asymptotic}. This entails the desired result since $H^2(\Omega )$ is compactly embedded in $H^s(\Omega )$ for all $s<2$.


\bigskip


\end{document}